\documentclass{amsart}
\usepackage{amssymb,amsthm}

\newcommand{\R}{\mathbf R}
\newcommand{\na}{\nabla}
\newcommand{\nablahat}{\widehat{\nabla}}
\newcommand{\nablan}{\nabla^N}
\renewcommand{\a}{\alpha}
\newcommand{\la}{\lambda}
\newcommand{\be}{\beta}
\newcommand{\ga}{\gamma}
\newcommand{\irho}{\rho^{-1}}
\newcommand{\ze}{Z^*}
\newcommand{\dij}{\delta_{ij}}

\renewcommand{\>}{\rangle}
\newcommand{\del}{\partial}

\newcommand{\id}{\operatorname{id}}
\newcommand{\tr}{\operatorname{tr}}
\newcommand{\grad}{\operatorname{grad}}

\newcommand{\nx}[2]{\nabla_{X_{#1}}X_{#2}}
\newcommand{\gam}[3]{\Gamma_{{#1}{#2}}^{#3}}
\newcommand{\hatgam}[3]{\widehat\Gamma_{{#1}{#2}}^{#3}}

\theoremstyle{definition}
\newtheorem{df}{Definition}[section]
\newtheorem{rem}[df]{Remark}

\theoremstyle{plain}
\newtheorem{theo}[df]{Theorem}
\newtheorem{prop}[df]{Proposition}
\newtheorem{lem}[df]{Lemma}

\numberwithin{equation}{section}

\begin{document}

\parindent=0mm
\parskip=3mm

\title[Quasi-umbilical hypersurfaces congruent to their centre
map]{Quasi-umbilical affine hypersurfaces congruent to their centre map}

\begin{abstract}
In this paper, we study strictly convex affine hypersurfaces centroaffinely
congruent to their centre map, in the case when the shape operator has two
distinct eigenvalues: one of multiplicity $1$, and one nonzero of
multiplicity $n-1$. We show how to construct them from $(n-1)$-dimensional
affine hyperspheres.
\end{abstract}

\author{A. J. Vanderwinden}
\address{Laboratoire de Math\'ematiques\\
Universit\'e de Valenciennes\\
Le Mont Houy -- B\^atiment ISTV2\\
59313 Valenciennes Cedex 9\\
France}
\email{ajvdwind@free.fr}
\date{June 1, 2012}

\maketitle

\section{Introduction}

In \cite{FV}, the authors introduced the notion of \emph{centre map} for a
centroaffine hypersurface and studied affine hypersurfaces centroaffinely
congruent to their centre map, completely solving the problem for
positive definite surfaces.

The solution to this problem is known in higher dimensions for
positive definite improper affine hyperspheres \cite{Tn} (i.e.\ for which the shape
operator $S$ identically vanishes), and for generic hypersurfaces \cite{Tg} (i.e.\ for which
$S$ has $n$ different, nonzero eigenvalues). In this paper, we investigate
the intermediate case of positive definite quasi-umbilical hypersurfaces,
i.e.\ when $S$ has two distinct eigenvalues: $\la_0$, of
multiplicity $1$, and $\la_1$ of multiplicity $n-1$.\\
More precisely, we prove the
\begin{theo}
 \label{resultat}
 Let $f:M^n\to \R^{n+1}$ be an affine immersion centroaffinely congruent to
its centre map $c$. Assume that both $f$ and $c$ are centroaffine, that the
Blaschke metric $h$ is positive definite, and that $f$ is quasi-umbilical,
with the
multiple eigenvalue $\la_1\neq 0$.
\\
Then such a hypersurface exists iff $\la_0+\la_1<0$, and in that case
 $(M,h)$ is locally isometric to a warped product
$\R\times_{e^F} N^{n-1}$. Moreover,
\begin{itemize}
\item if $(n+2)\la_0+n\la_1\neq 0$, then there exists a proper affine
hypersphere $g_2:N\to\R^n$ such that, up to an affine transformation of $\R^{n+1}$,
\begin{equation}
 \label{f}
 f(t,\vec u)=\left(\;t^{-2K_1} g_2(\vec u)\;,\;\frac{t^N}N\;\right),
\end{equation}
where $K_1$ and $N$ are constants related to the $\la_i$'s.
\newpage
\item if $(n+2)\la_0+n\la_1=0$, then, up to an affine transformation of
$\R^{n+1}$,
\begin{equation}
 f(t,\vec u)=\left(\;t^{-2K_1}\;,\;t^{-2K_1}\,\vec u\;,\;\varphi_0\,t^{-2K_1}\bigl(\mathcal
F(\vec u)-\frac1{2K_1}\log t\bigr)\;\right),
\end{equation}
where $\mathcal F$ is a solution of the Monge--Amp\`ere equation, and
$K_1,\varphi_0$ are constants.
\end{itemize}
The converse also holds.
\end{theo}

The hypersurfaces in Theorem~\ref{resultat} are similar to those described in
\cite{S}, where hypersurfaces with pointwise $SO(n-1)$-symmetry are studied.
The shape operator and difference tensor in that paper have
indeed the same form as the one we get under the assumptions of
Theorem~\ref{resultat}, the proof of which follows in part that of
\cite[Theorem 3.1]{S}.

{\bf Acknowledgements:} I am very grateful to Luc Vrancken for many valuable
discussions. 

\section{Preliminaries and notations}

Let us now very briefly recall some basic notions of affine geometry (see
\cite{NS} for details) and introduce the relevant notations.

Let $f:M\to \R^{n+1}$ be a non-degenerate immersion of an $n$-dimensional
oriented manifold $M$ into $\R^{n+1}$, with its Blaschke structure. Let us
denote by
\begin{itemize}
 \item $D$ the standard flat affine connection on $\R^{n+1}$,
 \item $\xi$ the affine normal of $f$,
 \item $\nabla$ the induced equiaffine connection on $M$,
 \item $h$ the equiaffine metric on $M$,
 \item $S$ the shape operator of $f$.
\end{itemize}

The above quantities are related by the following relations, for all vector
fields $X$ and $Y$ on $M$:
\begin{align*}
 D_X f_* Y&=f_*\nabla_X Y+h(X,Y)\xi,
 \\
 D_X \xi&=-f_* SX.
\end{align*}
(We will often drop the symbol $f_*$ in the sequel.)

The standard volume form $\det$ on $\R^{n+1}$ induces a volume form $\omega$
on $M$, defined as $\omega(X_1,\dots,X_n)=\det(X_1,\dots,X_n,\xi)$, and,
$\xi$ being the affine normal,
\[
 \omega^2(X_1,\dots,X_n)=\det(h_{ij}), \quad \mbox{where } h_{ij}=h(X_i,X_j).
\]

We will also denote by
\begin{itemize}
 \item $\nablahat$ the Levi-Civita connection of the metric $h$,
 \item $K$ the \emph{difference tensor}, defined by
 \[
  K(X,Y)=K_X Y=\nabla_X Y-\nablahat_X Y.
 \]
\end{itemize}
Recall \cite[Proposition~II.4.1]{NS} that
\begin{equation}
\label{KNablah}
h(K_X(Y),Z)=-\frac12(\nabla h)(X,Y,Z)
\end{equation}
and also that the apolarity condition $\nabla \omega=0$ can be expressed as $\tr K_X=0$ for
any vector field $X$ on $M$.

For all $u\in M$, the position vector $f(u)$ can be decomposed as 
\[
 f(u)=f_*Z_u+\rho(u)\xi_u,
\]
where $Z$ is a vector field on $M$ and $\rho$ the affine support function of
$f$.

We now recall the definition of the \emph{centre
map}, which has been introduced in \cite{FV}.

\begin{df}
 The \emph{centre map} of an immersion $f:M\to \R^{n+1}$ is the map $c:M\to
\R^{n+1}$
defined for all $u\in M$ by
\[
c(u)=f(u)-\rho(u)\xi_u=f_*Z_u.
\]
\end{df}

It follows that
\[
 c_*X=f_*(\id+\rho S)X-(X \rho)\xi,
\]
hence the centre map of an immersion $f$ is itself an immersion iff
\[
 \ker(\id+\rho S)\cap \ker d\rho=\{0\}.
\]

{}From now on, we will assume that the immersion $f$ is \emph{centroaffine}, i.e.\
that the position vector is everywhere transversal to the tangent space, and that
the centre map $c$ of $f$ is centroaffine, too, which amounts to
\begin{equation}
\label{ccentroaff}
\dim\<f_*Z_u^*\;,\;f_*(\id+\rho S)X-(X\rho)\xi_u\mid X\in T_uM\> = n+1,
\end{equation}
where we have used the notation $\ze=\irho Z$.

We are interested in immersions $f$ which are centroaffinely congruent to
their centre map $c$.

The following result has been established in \cite[Propositions~4.1,~4.2]{FV}:

\begin{prop}[Furuhata--Vrancken]
\label{congru}
Let $f:M\to \R^{n+1}$ be an affine immersion whose centre map $c$ is
a
centroaffine immersion. Then $f$ is centroaffinely congruent with $c$ iff
there exist a nowhere vanishing function $\rho$ and a vector field $\ze$ on $M$
satisfying the following system of equations for all vector fields $X,Y$ on $M$:
\begin{align}
 X(\rho)&=-\rho\, h(X,\ze)\label{ccf1},
 \\
 (\na_X S)Y&=h(X,\ze)SY+h(Y,\ze)SX-h(X,Y)S\ze\label{ccf2},
 \\
 (\na h)(X,Y,\ze)&=-2\irho h(X,Y)-2h(X,SY)-h(X,Y)h(\ze,\ze)\label{ccf3},
 \\
 \na_X \ze&=h(X,\ze)\ze+\irho X+SX\label{ccf4}.
\end{align}
\end{prop}

Using the apolarity condition, (\ref{KNablah}), and (\ref{ccf3}), we get
\begin{equation}
 \irho=-\frac 1 n \tr S-\frac 1 2 h(\ze,\ze)\label{rho},
\end{equation}
hence we can reformulate (\ref{ccf3}) as
\begin{equation}
 (\na h)(X,Y,\ze)=\frac 2 n \tr S\; h(X,Y)-2h(X,SY)\label{ccf3bis}.
\end{equation}

\section{Preliminary computations}

Let $f:M\to\R^{n+1}$ be an immersion whose centre map $c$ is itself a
centroaffine immersion, centroaffinely congruent to $f$.

We also assume that the metric $h$ induced by $f$ is positive definite. From
the Ricci equation, there exists a local $h$-orthonormal basis
$\{X_0,X_1,\dots,X_{n-1}\}$ of eigenvectors for the shape operator $S$.\\
If we denote by $\la_0,\dots,\la_{n-1}$ the corresponding eigenvalues, then
the Codazzi equation for $S$ in this basis reads:
\begin{equation}
  \label{CodaS}
  X_i(\la_j)X_j+\sum_{k=0}^{n-1} (\la_j-\la_k)\gam ijk X_k=
  X_j(\la_i)X_i+\sum_{k=0}^{n-1} (\la_i-\la_k)\gam jik X_k,
 \end{equation}
where $\gam ijk$ denote the Christoffel symbols of the equiaffine
connection $\na$ of $f$.

Writing $\ze=\sum_{i=0}^{n-1} a_iX_i$, we get from (\ref{ccf1}) that $X_i(\rho)=-\rho a_i$. By \cite[Proposition~4.3]{FV}, there exist constants $\nu_j$ such that $\rho\la_j=\nu_j$. Applying $X_i$ to this equality, we obtain
\begin{equation}
\label{lambda}
X_i(\la_j)=a_i\la_j.
\end{equation}

We now restrict to the quasi-umbilical case, i.e.\ when $S$ has two distinct
eigenvalues:
\begin{itemize}
\item $\la_0$, with eigenspace $\<X_0\>$,
\item $\la_1$, nonzero, with eigenspace $\<X_1,\dots,X_{n-1}\>$.
\end{itemize}
\vspace*{2mm}

For $i,j=1,\dots,n-1$, (\ref{CodaS}) now simplifies to 
\[
X_i(\la_1)X_j+(\la_1-\la_0)\gam ij0 X_0= X_j(\la_1)X_i+(\la_1-\la_0)\gam ji0
X_0.
\]
Therefore $X_i(\la_1)=0$ for $i=1,\dots,n-1$, so by (\ref{lambda}), $a_i=0$, i.e.
\[
\ze=a_0X_0.
\]

Let us now introduce the two constants
\[
 K_0=\frac{\la_0}{\la_0-\la_1}, \qquad  K_1=\frac{\la_1}{\la_0-\la_1}.
\]
\newpage

Using (\ref{CodaS}), the Codazzi equation for $h$, and the apolarity condition,
we get 
\begin{lem}
\label{gamma}
 For $i,j=1,\dots, n-1$, one has
\begin{align*}
 \nx 00&=-\frac{(n-1)}2 a_0(K_0+K_1) X_0,
 \\
 \nx 0i&=\frac{a_0}2 (K_0+K_1)X_i+\sum_{k\neq 0,i}\gam 0ik X_k,
 \\
 \nx i0&=a_0K_1X_i,
 \\
 \nx ij&=\dij a_0K_0X_0+\sum_{k=1}^{n-1}\gam ijk X_k.
\end{align*}
\end{lem}

{}From (\ref{ccf3}) we have, for $i=1,\dots,n-1$,
\begin{align*}
 -a_0^2(K_0+K_1)&=\nabla h(X_i,X_i,\ze)=-2\irho-2\la_1-a_0^2,
 \\
 (n-1)a_0^2(K_0+K_1)&=\nabla h(X_0,X_0,\ze)=-2\irho-2\la_0-a_0^2,
\end{align*}
and from (\ref{ccf4}),
\begin{align*}
(\irho+\la_1)X_i&=\nabla_{X_i} a_0X_0=X_i(a_0)X_0+a_0^2K_1X_i,
\\
\bigl(a_0^2+(\irho+\la_0)\bigr)X_0&=\nabla_{X_0}
a_0X_0=\left(X_0(a_0)-\frac{n-1}2\,
a_0^2(K_0+K_1)\right)X_0,
\end{align*}
so we deduce that 
\begin{align}
 X_0(a_0)&=\frac{a_0^2}2,
 \\
 \irho+\la_1&=a_0^2K_1\label{rhola1},
 \\
 a_0^2(\la_0+\la_1)&=-\frac 2n(\la_0-\la_1)^2\label{a0carre}.
\end{align}

\begin{rem}
Equation (\ref{a0carre}) shows that we must have $\la_0+\la_1<0$, as stated in Theorem~\ref{resultat}.
\end{rem}

\begin{lem}
\label{immcaf}
Under the assumptions of Theorem~\ref{resultat}, the centre map of $f$ is a
centroaffine immersion.
\end{lem}

\begin{proof}
We know from (\ref{ccentroaff}) that $c$ is a centroaffine immersion iff 
\[
\dim\<f_*Z_u^*\;,\;f_*(\id+\rho S)X-(X\rho)\xi_u\mid X\in T_u M\>=n+1
\]
iff the $n+1$ vectors
\[
a_0X_0\;,\;
(1+\rho\la_1)X_1\;,\;\dots\;,\;(1+\rho\la_1)X_{n-1}\;,\;(1+\rho\la_0)X_0+\rho a_0\xi
\]
are linearly independent iff $1+\rho\la_1\neq 0$.\\
If $\irho=-\la_1$, then we would get from (\ref{rho}) that $a_0^2=\frac 2 n
(\la_1-\la_0)$. This and (\ref{a0carre}) would imply that $\la_1=0$, a
contradiction.
\end{proof}
\newpage

A short computation using Lemma~\ref{gamma} leads to the following
\begin{lem}~
\label{warped}
\begin{itemize}
\item For $i,j\ge1$, $\nablahat_ {X_i}X_ j=\frac 12 \dij a_0X_0+\sum_{k=1}^{n-1}\hatgam ijk
X_k$
where $\hatgam ijk$ denote the Christoffel symbols of the
Levi-Civita connection $\nablahat$.
\item $\nablahat_{X_0}X_0=0$.
\item The difference tensor $K_{X_0}$ takes the form
\[
K_{X_0}=\begin{pmatrix}
\text{\footnotesize$-\frac{n-1}2\,a_0(K_0+K_1)$}&&0\qquad\dots\qquad0
\\[2mm]
0&&
\\
\vdots&&\frac12\,a_0(K_0+K_1)\id_{n-1}
\\[1mm] 
0&&
\end{pmatrix}.
\]
\end{itemize}
\end{lem}

\begin{rem}
{}From Lemma~\ref{warped}, we see that the form of $K_{X_0}$, as well as that
of the shape operator $S$, is the same as in \cite{S}.
\end{rem}

\section{Warped products}

Let $(M_1,g_1)$ and $(M_2,g_2)$ be two Riemannian manifolds. Using the
appropriate projections, any vector $V$ tangent to $M_1\times M_2$ can be
decomposed as $V=V_1+V_2$, with $V_i$ tangent to $M_i$ ($i=1,2)$.

Recall that the \emph{warped metric} $g_1\times_{e^F} g_2$ on $M_1\times
M_2$ is defined by
\[
g(V,W)=g_1(V_1,W_1)+e^{2F}g_2(V_2,W_2),
\]
where $F$ is a function on $M_1\times M_2$ depending only on $M_1$.

The manifold $M_1\times M_2$, endowed with this metric, is a Riemannian
manifold, denoted by $M_1\times_{e^F} M_2$.

We will now use the following special case of a theorem of N\"olker \cite{No}:
\begin{prop}
 Let $(M,g)$ be a Riemannian manifold with Levi-Civita connection $\nablahat$, whose tangent bundle splits
into two orthogonal distributions $\mathcal N_1$ and $\mathcal N_2$.
Assume that there exists $H\in \mathcal N_1$ such that for all $X,Y\in\mathcal N_1$, $U,V\in\mathcal N_2$,
one has
\begin{align*}
 \nablahat_X Y&\in \mathcal N_1,
 \\
g(\nablahat_U V, Z)&=g(U,V) g(H,Z)\qquad\mbox{for all $Z\in \mathcal N_1$}.
\end{align*}
Assume further that $U(|H|)=0$ for all $U\in \mathcal N_2$. Then $(M,g)$ is
locally isometric to a warped product $M_1\times_{e^F} M_2$, with $M_i$
integral manifolds of $\mathcal N_i$.\\
Moreover, one has $\grad F=-H$.
\end{prop}

So from Lemma~\ref{warped}, we get that the Riemannian manifold $(M,h)$ is
locally isometric to a warped product $\R\times_{e^F} N^{n-1}$, with the
induced
metric
$h_N$ on $N$ given by $h_N(X_i,X_j)=e^{-2F}\dij$ $(i,j=1,\dots,n-1)$, and
$H=\frac 12\,a_0X_0$.

We now choose coordinates local coordinates $u_1,\dots,u_{n-1}$ on $N$, and a local coordinate $t$ on $\R$ such that $X_0=\del_t$.

\section{Proof of Theorem~\ref{resultat}: case $(n+2)\la_0+n\la_1\neq 0$}

We construct two maps $g_i:M\to\R^{n+1}$ ($i=1,2$) of the form $g_i=\a_i\xi+\be_i X_0$ such that
$D_{X_i}g_1=D_{X_0} g_2=0$.

A straightforward computation using Lemma~\ref{gamma} and (\ref{a0carre}) leads to 
\begin{lem}
 \label{g1}
 The map $g_1=a_0K_1\xi+\la_1X_0$ satisfies
 \begin{align*}
  D_{X_i}g_1&=0\qquad (i=1,\dots,n-1),
  \\
  D_{X_0}g_1&=-\frac{a_0}2\,\bigl((n+1)K_1+(n-1)K_0\bigr)\,g_1.
 \end{align*}
\end{lem}

Hence, there exist a function $c(t)$ and a constant vector $C_0$ such that
$g_1(t)=c(t)C_0$.

\begin{lem}
 \label{g2}
 There exists a map $g_2=\a_2\xi+\be_2 X_0$ such that $D_{X_0}
g_2=0$.
\end{lem}

\begin{proof}
 Let us denote by $\nablan$ the restriction of $\nabla$ to
$\<X_1,\dots,X_{n-1}\>$.

For $i,j=1,\dots,n-1$, we have from Lemma~\ref{gamma}
\[
 D_{X_i} X_j=f_*\nablan_{X_i} X_j+\dij(a_0K_0X_0+\xi).
\]
The map $\phi=a_0K_0X_0+\xi$ satisfies 
\[
\phi_*X_i=D_{X_i} \phi=(a_0^2K_0K_1-\la_1)X_i=-\zeta(t)X_i
\]
and from (\ref{a0carre}) we also have
\[
\phi_*X_0=D_{X_0} \phi=a_0K_0\phi.
\]
Hence we can find a function $\a_2(t)$ with $D_{X_0}(\a_2\phi)=0$. This
function has to satisfy 
\begin{equation}
\label{alpha2}
X_0(\a_2)=-a_0K_0\a_2,
\end{equation}
so for the map $g_2=\a_2\phi$, we get
$D_{X_0} g_2=0$ and $D_{X_i} g_2=\eta(t)X_i$, with
$\eta=\a_2(a_0^2K_0K_1-\la_1)=-\a_2\zeta$.
\end{proof}

Notice for further use that by (\ref{a0carre}), 
\begin{equation}
 \label{zeta}
 \zeta=\frac{\la_1}n\left(\frac{(n+2)\la_0+n\la_1}{\la_0+\la_1}\right),
\end{equation}
hence the condition in the title of this section reads $\zeta\neq0$.

\begin{prop}
When $\zeta\neq 0$, the map $g_2$ is an immersion of $N$ as a proper affine
hypersphere in some
hyperplane of $\R^{n+1}$.
\end{prop}

\begin{proof}
 We have
\[
\begin{aligned}
 D_{X_j}D_{X_i}\ g_2&=\eta(t) D_{X_j} X_i
 \\
  &=\eta(t)[f_*\nablan_{X_j} X_i+\dij\phi]
 \\
  &=g_{2*}(\nablan_{X_j} X_i)+\dij \eta(t)\phi
 \\
  &=g_{2*}(\nablan_{X_j} X_i)-\dij \zeta(t) g_2.
\end{aligned}
\]
When $\zeta\neq 0$, $g_2$ can be viewed
as an immersion of $N$ into $\R^{n+1}$. The above computation shows that
$g_2$ actually lies in some fixed hyperplane of $\R^{n+1}$, namely $\mathcal
H=\<X_1(p), X_2(p),\dots,X_{n-1}(p),g_2(p)\>$ for some given point $p$. Hence
$g_2$ is an immersion of $N$ into $\mathcal H$, and the position vector is
transversal to $g_{2*}(N)$. From Lemma~\ref{gamma}, we see that the
difference
tensor $K^N$ satisfies the apolarity condition, hence $g_2$ is (possibly up to
a constant factor) the affine normal of $g_2$, which is therefore a
proper affine hypersphere in $\mathcal H$.
\end{proof}

\begin{rem}
When $\zeta\neq 0$, the vector field $g_1$ is transversal to $\mathcal H$.
\end{rem}

\begin{proof} One has
\begin{align*}
  a_0K_1\xi+\la_1 X_0&=\frac{\la}{\a_2}g_2+\sum_{i=1}^{n-1} a_i X_i
  \\
  \mbox{iff}\quad a_0K_1\xi+\la_1 X_0&=\la(\xi+a_0K_0X_0)+\sum_{i=1}^{n-1}a_i
X_i
\end{align*}
iff $a_i=0$ for $i=1,\dots, n-1$, $\la=a_0K_1$, and $\la_1=a_0^2K_0K_1$, i.e.\
$\zeta=0$.
\end{proof}

{}From $X_0 (a_0)=\frac{a_0^2}2$, we get $a_0=-\frac 2 t$.
Hence, (\ref{g1}) gives
\begin{align*}
 c'&=-\frac{a_0}2\,\bigl((n+1)K_1+(n-1)K_0\bigr)\,c
 \\
 &=\frac 1t\,\bigl((n+1)K_1+(n-1)K_0\bigr)\,c,
\end{align*}
whence 
\begin{equation}
 c(t)=n_1t^{(n+1)K_1+(n-1)K_0}\label{c}
\end{equation}
for some constant $n_1$.

Solving
\[
\left\{
\begin{aligned}
g_1&=a_0K_1\xi+\la_1X_0,
\\
g_2&=\a_2\xi+\a_2a_0K_0X_0
\end{aligned}
\right.
\]
for $X_0$, we get
\begin{equation}
X_0=\frac{a_0K_1}{\eta}g_2+\frac c{\zeta} C_0.
\end{equation}

Hence $\frac{\del f}{\del t}=\frac{a_0K_1}{\zeta}g_2+\frac c{\zeta}C_0$,
which, after an appropriate affine transformation (putting $C_0$ in the
$e_{n+1}$-direction), gives the following expression for $f$:
\[
 f(t,\vec u)=\bigl(\ga_1(t)g_2(\vec u),\ga_2(t)\bigr),
\]
where $\vec u=(u_1,\dots,u_{n-1})$ and
\[
 \ga_1(t)=\int\frac{a_0K_1}{\eta}(t)\,dt, \qquad\qquad
\ga_2(t)=\int\frac{c(t)}{\zeta(t)}\,dt.
\]

Let us now explicitly compute $\ga_1$ and $\ga_2$. 

By (\ref{lambda}), we know that the eigenvalues $\la_i$ only depend on
$t$, with
$\la_i'=-\frac 2 t \la_i$, hence $\la_i=\frac{l_i}{t^2}$ with $l_i$ constant.

Since $\zeta=\la_1-a_0^2K_0K_1$, we get $\zeta=\frac{\zeta_0}{t^2}$, where by
(\ref{zeta}), $\zeta_0=\frac{l_1}n\left(\frac{(n+2)l_0+nl_1}{l_0+l_1}\right)$.

Using (\ref{c}), we have
\[
 \frac c{\zeta}(t)=\frac {n_1}{\zeta_0}t^{(n-1)K_1+(n+1)K_0}.
\]
Notice that $(n-1)K_1+(n+1)K_0\neq -1$. Otherwise, $K_0+K_1=-\frac 2 n$,
hence, by (\ref{a0carre}), $a_0^2=\la_0-\la_1$, i.e.\ $a_0^2K_1=\la_1$. But
by (\ref{rhola1}), $a_0^2K_1=\irho+\la_1$, a contradiction.
So we get
\[
\ga_2(t)=\frac{n_1}{\zeta_0}\,\frac{t^N} N,
\]
where $N=(n-2)K_1+(n+2)K_0\neq 0$.

On the other hand, $\eta=-\a_2\zeta$, where, from (\ref{alpha2}),
$\a_2=n_2t^{2K_0}$, with $n_2$ constant.
Hence $\eta=-n_2\zeta_0t^{2(K_0-1)}=-n_2\zeta_0t^{2 K_1}$.

It is easy to check that $\eta'=-a_0K_1\eta$, hence $\ga_1=\frac 1
\eta=-\frac1{n_2\zeta_0}t^{-2K_1}$. So we have
\begin{equation}
 \label{explicite}
 f(t,\vec u)=\left(-\frac1{n_2\zeta_0}t^{-2K_1}g_2(\vec u)\;,\;
\frac{n_1}{\zeta_0} \frac{t^N}N\right).
\end{equation}

Let us now check that the hypersurfaces described in (\ref{explicite}) do indeed satisfy
the assumptions of Theorem~\ref{resultat}.

One has 
\begin{align*}
 \del_t&=(\ga_1' g_2,\ga_2'),
 \\
 \del_{u_i}&=\bigl(\ga_1 g_{2*}(\del_{u_i}),0\bigr),
 \\
 \xi&=\frac 1{\a_2} g_2-a_0K_0\del_t,
\end{align*}
so that
\begin{align*}
 D_{\del_t}\del_t&=\bigl(\frac{\ga_2''}{\ga_2'}+a_0 K_0\bigr)
\del_t+\xi,
\\
D_{\del_t}\del_{u_i}&=\frac{\ga_1'}{\ga_1}\del_{u_i},
\\
D_{\del_{u_j}}\del_{u_i}&=\nabla^N_{\del_{u_j}}\del_{u_i}+e^{2F}h_N(\del_{u_i}
, \del_ { u_j })(a_0K_0\del_t+\xi).
\end{align*}
Hence
\begin{align*}
 h(\del_t,\del_t)&=1,
 \\
 h(\del{u_i},\del{u_j})&=e^{2F}h_N(\del{u_i},\del{u_j}),
 \\
 h(\del_t,\del{u_i})&=0,
\end{align*}
with $h_N$ the positive definite metric induced on $N$ by $g_2$.

We see that $h$ is positive definite and that $\det
h= e^{2(n-1)F} \det h_N$. 
\\
On the other hand,
\begin{align*}
\omega(\del_t\;,\;\del{u_1}\;,\;\dots\;,\;\del_{u_{n-1}})
&=\det\left(\del_t\;,\;\del{u_1}\;,\;\dots\;,\;\del_{u_{n-1}}\;,\;\frac
1{\a_2}g_2-a_0K_0\del_t\right)
\\
&=\det\left(\del_t\;,\;\del{u_1}\;,\;\dots\;,\;\del_{u_{n-1}}\;,\;\frac
1{\a_2}g_2\right)
\\
&=\begin{vmatrix}
&&&&
\\
\frac{a_0K_1}\eta\,g_2&\frac1\eta\,g_{2*}(\del_{u_1})&\dots&\frac1\eta\,g_{2*}
(\del_{u_{n-1}})&\frac1{\a_2}\,g_2
\\
&&&&
\\
\dfrac c\zeta&0&\dots&0&0
\end{vmatrix}
\\[2mm]
&=(-1)^{n+2}\frac c{\zeta}\;
\det\left(\frac 1{\eta}\,g_{2*}(\del_{u_1})\,,\,\dots\,,\,\frac1{\eta}
g_{2*}(\del_{u_{n-1}})\,,\,\frac 1{\a_2}g_2\right)
\\[2mm]
&=(-1)^n\frac c{\a_2\zeta\eta^{n-1}}\;
\det\Bigl(g_{2*}(\del_{u_1})\;,\;\dots\;,\;g_{2*}(\del_{u_{n-1}})\;,
\;g_2\Bigr)
\\[2mm]
&=(-1)^{n+1}\frac c{\eta ^n}\, \sqrt {\det h_N}.
\end{align*}
For $\xi$ to be the affine normal, we have to check that
\begin{equation}
\label{volume}
\omega^2(\del_t,\del_{u_1},\dots,\del_{u_{n-1}})=\det
h=e^{2(n-1)F}\det h_N.
\end{equation}
Since $\grad f=-\frac{a_0}2\,X_0$, $e^F=e_0 t$
for some constant $e_0$, and (\ref{volume}) reads:
\begin{equation}
\label{condition}
 \frac{c^2}{\eta^{2n}}=\frac{n_1^2}{n_2^{2n}\zeta_0^{2n}}\,t^{2(n-1)}=\,e_0^{
2(n-1)}t^{2(n-1)},
\end{equation}
which does hold after adjusting the integration constants $n_1, n_2, e_0$.

A straightforward computation shows that $D_{X_0}\xi=-\la_0 X_0$ and
$D_{X_i}\xi=-\la_1 X_i$ for $i=1,\dots,n-1$.

Let us now check that $f$ is indeed congruent to its centre
map $c_f$. By
definition, $c_f=f_*Z=\rho f_*\ze$. From (\ref{ccf1}) we deduce
$\rho=\rho_0t^2$, with $\rho_0$ a constant. So
\begin{align*}
 c_f&=-2\rho_0tX_0
 \\
 &=\left(\frac{4\rho_0K_1}{\eta} g_2,-2\rho_0\frac{n_1}{\zeta_0} t^N\right).
\end{align*}
On the other hand, by (\ref{explicite})
\[
f=\left(\frac 1 {\eta}g_2,\frac{n_1}{\zeta_0}\frac{t^N}N\right),
\]
hence $c_f=Af$, with
\[
A=\begin{pmatrix}
&&0
\\
&4\rho_0K_1\,\id_n&\vdots
\\
&&0
\\[1.5mm]
&0\quad\dots\quad0&-2\rho_0N
\end{pmatrix}.
\]

\section{Proof of Theorem~\ref{resultat}: case $(n+2)\la_0+n\la_1=0$}

In this case, we have $\zeta=0$ (cf.\ (\ref{zeta})).

As in the case $\zeta\neq 0$, we have, for $i,j=1,\dots,n-1$,
\begin{equation}
 \label{derfij}
 D_{X_i} X_j=f_*\nablan_{X_i} X_j+\dij(a_0K_0X_0+\xi).
\end{equation}

The map $\phi=a_0K_0X_0+\xi$ satisfies 
\[
 \phi_*X_i=D_{X_i} \phi=0
\]
and 
\[
 \phi_*X_0=D_{X_0} \phi=a_0K_0\phi=-\frac 2t K_0\phi,
\]
so that $\phi=t^{-2K_0} \phi_0$ with $\phi_0$ a constant vector.

Since $(n+2)K_0+nK_1=0$, one has
\begin{equation}
\label{derftt}
 D_{X_0}X_{0}=\frac{a_0}2(K_0+K_1)X_0+\phi,
\end{equation}
hence $f(t,\vec u)$ takes the form
\[
 f(t,\vec u)=g_0(\vec u)\ga_1(t)+g_1(\vec u)1+\a(t)\phi_0
\]
and
\[
 X_0=\ga_1'(t)\,g_0(\vec u)+\a'(t)\phi_0.
\]

{}From $D_{\del{u_i}} X_0=a_0K_1\del{u_i}$, we deduce
\begin{itemize}
\item $\ga_1'=a_0K_1\ga_1$, i.e.\ $\ga_1=\ga_0t^{-2K_1} $($\ga_0$
constant),
\item $g_1(\vec u)$ is constant,
\end{itemize}
and $\a(t)\phi_0$ is a solution of (\ref{derftt}) iff $\a(t)$ satisfies
\[
\a''(t)=\frac{a_0}2(K_0+K_1)\a'(t)+t^{-2K_0},
\]
i.e.
\[
 \a''(t)+\frac1t(K_0+K_1)\a'(t)=t^{-2K_0}.
\]
The general solution to this equation is
\[
 \a(t)=-\frac{t^{-2K_1}}{2K_1}\log
t-B\frac{t^{-2K_1}}{2K_1}+C,
\]
where $B$ and $C$ are constants. Hence, up to a translation, 
\begin{equation}
\label{fnul}
f(t,\vec u)=\ga_1 g_0(\vec u)-\frac{t^{-2K_1}}{2K_1}(\log t+B)\phi_0
\end{equation}
and
\[
X_0=a_0K_1\ga_1\,g_0(\vec u)+t^{-(K_0+K_1)}(\log t +B)\phi_0.
\]
\newpage

We now show that $g_0$ is an improper affine hypersphere in some hyperplane
of $\R^{n+1}$.

We first show that the $n+1$ vectors $g_0(\vec u),\del_{u_i}g_0(\vec
u),\phi_0$ 
$(i=1,\dots,n-1)$ are linearly independent.
Indeed, denoting by $(h_{ij})=h(\del_{u_i},\del_{u_j})$ ($i,j=0,\dots,n-1$),
we know that
$\det(\del_t,\del_{u_1},\dots,\del_{u_i},\xi)=\sqrt{\det (h_{ij})}\neq 0$.

Since $\xi=\phi-a_0K_0\del_t$ and $\del_{u_i}=\ga_1\del_{u_i}g_0$,
\begin{align*}
\det(\del_t,\del_{u_1},\dots,\del_{u_{n-1}},\xi)
&=\det(\del_t\,,\,\ga_1\del_{u_1}g_0\,,\,\dots\,,\,\ga_1 \del_{u_{n-1}}g_0\,,\,\phi)
\\
&=\det(a_0K_1\ga_1g_0\,,\,\ga_1\del_{u_1}g_0\,,\,
\dots\,,\,\ga_1\del_{u_{n-1}}g_0\,,\,t^{-2K_0}\phi_0),
\end{align*}
hence $\det(g_0,\del_{u_1}g_0, \dots ,\del_{u_{n-1}}g_0,\phi_0)\neq 0$.

Let us now fix a point $p_0$ in $N$ and choose a frame in $\R^{n+1}$
such that
\begin{align*}
g_0(p_0)=&(1,0,\dots,0),
\\[1mm]
\del_{u_i}g_0(p_0)=&(0,\dots,1,\dots,0),\qquad i=1,\dots,n-1,
\\[-1mm]
&\text{\footnotesize($1$ in $(i+1)$st position)}
\\[1mm]
\phi_0=&(0,\dots,0,\varphi_0)\qquad\text{($\varphi_0$ a constant)}.
\end{align*}
{}From (\ref{derfij}),
\begin{equation}
\label{dergoij}
D_{\del_{u_j}}D_{\del_{u_i}}g_0=g_{0*}\nabla^N_{\del_{u_j}}\del_{u_i}
+h^N(\del_ { u_i},\del_{u_j})\phi_0.
\end{equation}
This equation has a unique solution satisfying the initial conditions
$g_0(p_0)$ and $\del_{u_i}g_0(p_0)$. Looking at the first
component of (\ref{dergoij}), we see that $g_0$ lies in the hyperplane
$\mathcal H\equiv x_0=1$.

A straightforward computation shows that, since $(n-2)K_0+nK_1=0$,
\[
\omega(\del_{u_1}g_0,\dots,\del_{u_{n-1}}g_0,\phi)=\sqrt{\det h_N}.
\]
Moreover, $D_{\del_{u_i}}\phi=0$, hence $g_0$ is an improper affine hypersphere in
$\mathcal H$, with affine normal $\phi$. It is well kwown that any such map
is locally the graph of a function $\mathcal F:N\to\R$ solution of the
Monge--Amp\`ere equation $\det\left(\frac{\del^2\mathcal
F}{\del_{u_i}\del_{u_j}}\right)=1$, so that by (\ref{fnul}),
\begin{equation}
f(t,\vec u)
=\left(\;t^{-2K_1}\;,\;t^{-2K_1}\,\vec u\;,\;
\varphi_0\,t^{-2K_1}\bigl(\mathcal F(\vec u)-\frac1{2K_1}\log t\bigr)\;\right).
\end{equation}
We also have
\[
X_0=\left(\;\ga_1'\;,\;\ga_1'\vec u\;,\;
\varphi_0\left(\ga_1'\mathcal F(\vec u)+2K_1t^{-2K_1-1}\,\frac{\log t}{2K_1}-\frac{t^{-2K_1-1}}{2K_1}\right)\;\right),
\]
with $\ga_1'=-\frac 2t\,K_1\,t^{-2K_1}$.

Recall that the centre map is given by  $c_f=-2\rho_0tX_0$, hence $c_f=Af$ with 
\[
A=\begin{pmatrix}
\text{\footnotesize$4\rho_0K_1$}&0&\dots&\dots&0
\\
0&\text{\footnotesize$4\rho_0K_1$}&\ddots&&\vdots
\\
\vdots&\ddots&\ddots&\ddots&\vdots
\\
0&&\ddots&\text{\footnotesize$4\rho_0K_1$}&0
\\
\frac{\rho_0\varphi_0}{K_1}&0&\dots&0&\text{\footnotesize$4\rho_0K_1$}
\end{pmatrix}.
\]

\end{document}